\documentclass{article}[11pt]
\usepackage{amsthm,amsmath}
\usepackage{amssymb}
\input xypic
\usepackage[all,tips]{xy}
\newtheorem{theo}[subsection]{Theorem}
\newtheorem{lem}[subsection]{Lemma}
\newtheorem{prop}[subsection]{Proposition}
\newtheorem{corr}[subsection]{Corollary}

\newtheorem{rem}[subsection]{Remark}
\newcommand{\Q}{{\bf Q}}

\newcommand{\cB}{{\cal B}}
\newcommand{\cA}{{\cal A}}
\newcommand{\cC}{{\cal C}}

\newcommand{\cG}{{\cal G}}

\newcommand{\h}{{\mathfrak h}}
\newcommand{\g}{{\mathfrak g}}
\newcommand{\gl}{{\mathfrak l}}
\newcommand{\Z}{{\bf Z}}
\newcommand{\C}{{\bf C}}

\newcommand{\Rep}{{\cal R}{\it ep}}
\newcommand{\Vect}{{\cal V}{\it ect}}

\newcommand{\da}{\mbox{-}}
\newcommand{\se}{\mbox{:}}

\newcommand{\cro}{{\scriptscriptstyle {\times}}}
\newcommand{\ort}{{\scriptscriptstyle {\perp}}}
\newcommand{\di}{{\scriptscriptstyle {\#}}}

\date{}

\begin{document}
\title{Commutative algebras in Drinfeld categories of abelian Lie algebras.}
 
\author{Alexei Davydov and Vyacheslav Futorny\\ \\Max Planck Institut f\"ur Mathematik, \\
Vivatsgasse 7,
53111 Bonn, Germany\\ 
\\ Institute of Mathematics and Statistics, University of S\~ao Paulo,\\  CEP 05315-970, S\~ao Paulo, Brazil}
\maketitle

%\begin{center}
%${}^{1)}$ \\
%${}^{2)}$
%\end{center}
\begin{abstract}
We describe (braided-)commutative algebras with non-degenerate multiplicative form in certain braided monoidal categories, corresponding to abelian metric Lie algebras (so-called Drinfeld categories).
We also describe local modules over these algebras and classify commutative algebras with finite number of simple local modules. 
\end{abstract}

\tableofcontents
\section{Introduction}

Motivated by applications for representation theory of vertex operator algebras in this paper we systematically study commutative algebras and their local modules in Drinfeld categories of abelian metric Lie algebras.

In \cite{dr} Drinfeld associated to a non-degenerate invariant bilinear form ({\em metric}) on a Lie algebra an infinitesimal deformation of the canonical tensor structure on its representation category. This infinitesimal deformation is no longer a symmetric monoidal category. The deformed commutativity constraint is only a braiding. It was also explained in \cite{dr} how these ({\em Drinfeld}) categories are related to representation categories of the quantisations of Lie algebras. 

It is well-known that a metric on a Lie algebra gives rise to an {\em affinisation} of (a central extension of the Lie algebra of Laurent polynomials with coefficients in) the original Lie algebra. The category of representations of the affinisation has a tensor structure given by the so-called {\em fusion product}. It was explained in \cite{kl} how the induction functor links the Drinfeld category of a simple Lie algebra (with the Cartan-Kiling metric) with the representation category of its affinisation. In this case the infinitesimal deformation form \cite{dr} becomes a global one (on a certain subcategory of the Drinfeld category). 
Another class of metric Lie algebras for which the infinitesimal deformation form \cite{dr} becomes global is the class of abelian metric Lie algebras. Corresponding Drinfeld categories are related to categories of modules over Heisenberg vertex operator algebras \cite{ll}.

It is known (see \cite{ko}) that commutative algebras in the representation category of a vertex operator algebra correspond to extensions of this vertex operator algebra. Moreover the category of representations of an extended algebra coincides with the category of so-called {\em local} modules over the corresponding commutative algebra.

Here we study commutative algebras and their local modules in Drinfeld categories of abelian Lie algebras. After recalling basic facts about commutative algebras in braided monoidal categories (section \ref{combrloc}) and Drinfeld categories (section \ref{drcat}) we classify  commutative algebras which posses a non-degenerate bilinear form, compatible with the multiplication  and have trivial invariants (section \ref{comdr}). We prove (theorem \ref{coa}) that such algebras correspond to subgroups in the abelian metric Lie algebra, such that the restriction of the metric is integer valued and even.
Then we turn to local modules over commutative algebras (section \ref{locmod}). We show that the category of local modules has a grading, compatible with the tensor product, (proposition \ref{grad}) and characterise the trivial component of this grading (proposition \ref{dz}). We also construct invertible modules sitting in every non-trivial graded component (proposition \ref{fun}). All this together with some technical tools developed in the appendix allows us to classify commutative algebras with finite number of simple local modules (section \ref{finloc}). 

If not stated otherwise all linear algebra constructions will be assumed linear over the ground field $k$ which is algebraically closed  of characteristic zero. 

\section*{Acknowledgment}
The work on the paper began during A.D.'s visit to S\~ao Paulo, where the author was accommodated by the Institute of Mathematics and Statistics of the University of S\~ao Paulo. The paper was finished while A.D. was visiting Max Planck Institut f\"ur Mathematik (Bonn). The first author would like to thank these institutions for hospitality and excellent working conditions. The first author would like to thank FAPESP (grant no. 2008/10526-1) and Max Planck Gesellschaft, which financial support made the visit to S\~ao Paulo and Bonn possible. The second author was  supported in part by the CNPq grant (301743/2007-0) and by the Fapesp grant (2005/60337-2).

\section{Commutative algebras in braided categories and their local modules}\label{combrloc}

This is a preliminary section, where we recall basic facts about commutative algebras in braided monoidal categories and their modules. 

An (associative, unital) {\em algebra} in a monoidal category $\cC$ is a triple $(A,\mu,\iota)$ consisting of an
object $A\in\cC$ together with a {\em multiplication} $\mu:A\otimes A\to A$ and a {\em unit} map $\iota:1\to A$,
satisfying {\em associativiy} $$\mu(\mu\otimes I) = \mu(I\otimes\mu),$$ and {\em unit} $$\mu(\iota\otimes I) = I =
\mu(I\otimes\iota)$$ axioms. Where it will not cause confusion we will be talking about an algebra $A$, suppressing its
multiplication and unit maps.

A left {\em module} over an algebra $A$ is a pair $(M,\nu)$, where $M$ is an object of $\cC$ and $\nu:A\otimes M\to M$
is a morphism ({\em action map}), such that $$\nu(\mu\otimes 1) = \nu(1\otimes\nu).$$ A {\em homomorphism} of left
$A$-modules $M\to N$ is a morphism $f:M\to N$ in $\cC$ such that $$\nu_N(1\otimes f) = f\nu_M.$$
Left modules over an algebra $A\in\cC$ together with module homomorphisms form a category ${_A}{\cC}$. 

Now let $\cC$ be a braided monoidal category with the braiding $c_{X,Y}:X\otimes Y\to Y\otimes X$ (see \cite{js} for definition). An algebra $A$ in $\cC$ is {\em commutative} if $\mu c_{A,A} = \mu$. 

It was shown in \cite{pa} that the category ${_A}{\cC}$ of left modules over a commutative algebra $A$ is monoidal with respect to the tensor product $M\otimes_AN$ over $A$, which can be defines as a coequaliser
$$\xymatrix{M\otimes_AN & M\otimes N \ar[l] && M\otimes A\otimes N \ar@/^5pt/[ll]^{(\nu_M1)(c_{M,A}1)} \ar@/_5pt/[ll]_{1\nu_N} }$$

A (right) module $(M,\nu)$ over a commutative algebra $A$ is {\em local} if and only if the diagram
$$\xymatrix{M\otimes A  \ar[r]^\nu \ar[d]_{c_{M,A}} & M\\ A\otimes M \ar[r]^{c_{A,M}} & M\otimes A \ar[u]_\nu}$$
commutes. Denote by $\cC_A^{loc}$ the full subcategory of $\cC_A$ consisting of local modules. The following result was
established in \cite{pa} (see also \cite{ko}).
\begin{prop}
The category $\cC_A^{loc}$ is a full monoidal subcategory of ${_A}{\cC}$. Moreover, the braiding in $\cC$ induces a braiding in
${_A}{\cC}^{loc}$.
\end{prop}

We call a commutative algebra $A$ in a balanced category $\cC$ {\em balanced} if $\theta_A=1_A$, where $\theta$ is the balancing twist in $\cC$. It was proved in \cite{ko} (see also \cite{ffrs}) that the balancing twist $\theta_M$ of a local module $M$ over a balanced algebra $A$ is a homomorphism of $A$-modules, i.e. the category $\cC_A^{loc}$ of local modules over a balanced algebra is naturally balanced.

\section{Drinfeld categories}\label{drcat}

This is also a preliminary section where we recall the construction of our categories of interest, the so-called Drinfeld categories, associated to metric (and, more generally, Casimir) Lie algebras. 

A Lie algebra $\g$ is called {\em Casimir} if it is equipped with a $\g$-invariant symmetric bi-tensor $\Omega\in\g^{\otimes 2}$: $$[x\otimes 1+1\otimes x,\Omega]=0.$$

A finite dimensional Lie algebra $\g$ is {\em metric} if it is equipped with a non-degenerate symmetric bilinear form $(\da,\da)\se\g\otimes\g\to k$, which is $\g$-invariant $$([x,y],z) + (y,[x,z])) = 0,\quad x,y,z\in\g.$$ We will denote by $\Omega\in\g\otimes\g$ its {\em Casimir} element, i.e. unique element with the property 
$$\sum_i \Omega^{1}_i (\Omega^2_i,x) = x = \sum_i (x,\Omega^{1}_i)\Omega^2_i,$$ where $\sum_i \Omega^{1}_i\otimes\Omega^2_i = \Omega$. In particular $\Omega$ can be written as $\sum_i u_i\otimes u_i$ where $u_i$ is an orthonormal bases of $\g$, i.e. $(u_i,u_j) = \delta_{i,j}$. 
Note that a metric Lie algebra is a Casimir Lie algebra.

For example, a simple Lie algebra has a unique up to a scalar metric structure, given by the Cartan-Killing form (the trace form in the adjoint representation). Another example comes from an abelian Lie algebra with a non-degenerate symmetric bilinear from (which in this case is automatically invariant). It can be shown that a general metric algebra is an orthogonal sum of a semisimple Lie algebra and a solvable Lie algebra, and that a solvable metric Lie algebra is an iterated double extension of a one dimensional metric Lie algebra (see \cite{ml} for details). 

It was shown in \cite{dr} that over the formal power series $k[[h]]$ the category of representations $\Rep(\g)$ of a Casimir Lie algebra $\g$ can be equipped with a structure of braided monoidal category, where the tensor product is the original tensor product of representations ($\g$-modules) and the braiding given by
\begin{equation}\label{comd}
c_{M,N}(m\otimes n) = e^{\pi i h\Omega}(n\otimes m),\quad m\in M, n\in N,\quad M,N\in\Rep(\g).
\end{equation}
It turns out, that for the coherence axioms for braiding to work, one needs to deform the associativity constraint as well:
\begin{equation}\label{asd}
\alpha_{L,M,N}(l\otimes(m\otimes n)) = \Phi((l\otimes m)\otimes n),\quad l\in L, m\in M, n\in N,\quad L, M,N\in\Rep(\g).
\end{equation}
A solution for $\Phi$ (the so-called {\em Drinfeld associator}) was found in \cite{dr}, which is a formal (non-commutative) power series in $h\Omega_{12}$ and $h\Omega_{23}$. Here $\Omega_{12} = \Omega\otimes 1$ and $\Omega_{23} = 1\otimes \Omega$ are elements of $U(\g)^{\otimes 3}$. Moreover, it was shown in \cite{dr} that $\Phi(h\Omega_{12},h\Omega_{23})$ is the exponent of a formal power Lie series in $h\Omega_{12}$ and $h\Omega_{23}$, in particular it depends only on iterated commutators of $h\Omega_{12}$ and $h\Omega_{23}$.
\newline
To be precise it is not enough to work with $k[[h]]$-modules with $\g$-action. To make formulas (\ref{comd}),(\ref{asd}) work we need to consider the category $\Rep_h(\g)$ of modules over $U(\g)[[h]]$. We denote by $\cC_h(\g,\Omega)$ the category $\Rep_h(\g)$ with the braided monoidal structure given by (\ref{comd}),(\ref{asd}). We will call it {\em Drinfeld category} corresponding to a Casimir Lie algebra $\g$. Note that the Drinfeld category $\cC_h(\g,\Omega)$ comes equipped with a balancing structure, with the balancing twist defined by
$$\theta_M(m) = e^{\pi i h\omega}m,\quad m\in M,\quad M\in\Rep(\g).$$ Here $\omega=\mu(\Omega)\in U(\g)$ is the quadratic Casimir (the image of $\Omega$ under the multiplication map $\mu:U(\g)\otimes U(g)\to U(\g)$ of the universal enveloping algebra of $\g$. Indeed, the identity 
$$\Delta(\omega) - \omega\otimes 1 - 1\otimes\omega = 2\Omega$$ implies the balancing axiom.

It happens quite often that (after suitably redefining $\Rep(\g)$) the formal parameter $h$ in Drinfeld category $\cC_h(\g,\Omega)$ can be specialized to a number $c\in k$. For example if $\g$ is a simple Lie algebra then, restricting to the category $\Rep_{fd}(\g)$ of finite dimensional $\g$-modules, one can set $h$ to be any non-rational number $c\in k$ and define braided monoidal category $\cC_c(\g,\Omega)$. Moreover the category $\cC^{fd}_c(\g,\Omega)$ is equivalent to the category of finite dimensional representation of the quantum universal enveloping algebra $U_q(\g)$, where $q=e^{\frac{\pi i}{c}}$ \cite{kl}. 

Another series of examples is provided by nilpotent metric Lie algebras. In that case Drinfeld associator is the exponent of a Lie polynomial in $h\Omega_{12}$ and $h\Omega_{23}$. As well as the phase in (\ref{comd}) the action of this exponent is well defined on finite dimensional $\g$-modules, and hence on modules, with every cyclic submodule being finite dimensional. The following provides an intrinsic characterisation of such $\g$-modules.

We call a module $V$ over a Lie algebra $\g$ {\em locally finite} if for any $x\in\g$ and for any $v\in V$ there is a polynomial $f(t)$ such that $f(x)v=0$. 
\begin{lem}\label{lf}
A $\g$-module $M$ is locally finite if and only if for any $m\in M$ the cyclic submodule $U(\g)m$ is finite dimensional. 
\end{lem}
\begin{proof}
If $U(\g)m$ is finite dimensional then so is $\g m$. In particular a linear combination of $x^im$ is zero.

Conversely, choose a basis $\{x_1,...,x_n\}$ in $\g$. 
By Poincare-Birkhoff-Witt theorem $U(\g)m$ is spanned by $\{x_1^{i_1}...x_n^{i_n}m| i_1,...,i_n\in\Z_{\geq 0}\}$. We are going to show that $i_j,j=1,...,n$ are all bounded. 
By local finiteness of a $\g$-module $M$ the span of $\{x_n^{i_n}m\}$ is finite dimensional, i.e. $i_n$ is bounded say $i_n\leq s_n$. For any $i_n=1,...,s_n$ the span $\{x^{i_{n-1}}_{n-1}x_n^{i_n}m\}$ is again finite dimensional. By continuing this argument we get that $\{x_1^{i_1}...x_n^{i_n}m| i_1,...,i_n\in\Z_{\geq 0}\}$ is finite dimensional. 
\end{proof}
In particular, tensor product of locally finite modules is locally finite.
For a nilpotent Lie algebra $\g$ we will denote by $\cC_c(\g,\Omega)$ the category of locally finite $\g$-modules with the braided monoidal structure given by (\ref{comd}),(\ref{asd}) where $h$ is replaced by $c\in k^\cro$. Note that $\cC_c(\g,\Omega) = \cC_1(\g,c\Omega)$. In particular, for abelian $\g$ the category $\cC_c(\g,\Omega)$ does not depend on $c$ (since all metric structures on an abelian Lie algebra are equivalent). Thus for abelian $\g$ we will denote $\cC_c(\g,\Omega)$ simply by $\cC(\g,\Omega)$. 

\section{Commutative algebras in Drinfeld categories of abelian Lie algebras}\label{comdr}

Let $\h$ be an abelian metric Lie algebra with the non-degenerate form $(\da,\da)$ and the Casimir $\Omega$. Let $\cC(\h,\Omega)$ be the corresponding Drinfeld category, i.e. the category of locally finite $\h$-modules with ordinary tensor product and associativity constraint and with the braiding defined by 
\begin{equation}\label{br}
c_{M,N}(m\otimes n) = e^{\pi i\Omega}(n\otimes m),\quad m\in M, n\in N. 
\end{equation}
Any locally finite $\h$-module $V$ can be written as a sum
$$V = \bigoplus_{x\in\h}V_x,$$
where $V_x = \{v\in V|\ (y-(y,x)1)^{n_y}v=0,\ \forall y\in\h\}$ is a generalised eigenspace with the character $(x,\da)\in\h^*$. Indeed, 
for any $v\in V$, the space $\tilde{V}=U(\g)v$ is finite dimensional by lemma~\ref{lf}. Hence, $\tilde{V}\subset \bigoplus_{x\in\h}V_x\subset V$, and the statement follows.

We call the subset 
$$\gl(V) = \{x\in\h|\ V_x\not=0\}\subset\h$$ the {\em support} of $V$. 

The following is a strictification of lemma (\ref{lf}).
\begin{lem}\label{gi}
For $V,U\in\cC(\h,\Omega)$ 
$$(V\otimes U)_x = \bigoplus_{y+z=x}V_y\otimes U_z.$$
\end{lem}

\begin{proof}
By lemma~\ref{lf} we have
$$V\otimes U = \bigoplus_{x\in\h}(V\otimes U)_x.$$ On the other hand, 
$V = \bigoplus_{y\in\h}V_y$, $U = \bigoplus_{z\in\h}U_z$ and hence
$$V\otimes U=\bigoplus_{y,z\in\h}V_y\otimes U_z.$$ Taking into account that for any $v\in V$, $u\in U$, $x, h\in \h$ and any $y,z\in \h$ such that $x=y+z$ we have  
$$(h-(x,h)1)^N(v\otimes u)=\sum_{k=0}^N \left( {\begin{array}{*{20}c} N \\ k \\ \end{array}} \right)(h-(y,h)1)^k v\otimes (h-(z,h)1)^{N-k} u,
$$
we conclude that $V_y\otimes U_z\subset (V\otimes U)_x$.  Since the subspaces $(V\otimes U)_x$ and $(V\otimes U)_{x'}$ do not intersect if $x\neq x'$, the statement follows. 
\end{proof}

An algebra $A$ in the category $\cC(\h,\Omega)$ is an associative algebra with an $\h$-action by derivations. It follows from lemma (\ref{gi}) that the decomposition 
\begin{equation}\label{algg}
A = \bigoplus_{x\in\gl(A)}A_x
\end{equation}
into the sum of generalised eigenspaces is an algebra grading, i.e. $A_xA_y\subset A_{x+y}$. 
\newline
A bilinear form $\beta:A\otimes A\to k$ on an algebra $A$ is called {\em multiplicative} if 
$$\beta(ab,c) = \beta(a,bc),\quad \forall a,b,c\in A.$$

\begin{prop}
Let $A$ be an algebra in the category $\cC(\h,\Omega)$ with a non-degenerate multiplicative bilinear form and with the trivial subalgebra of invariants $A^\h=k$. Then $\gl=\gl(A)$ is a subgroup of $\h$ and $A$ is isomorphic to a skew group algebra $k[\gl,\alpha]$ for some 2-cocycle $\alpha:\gl\times\gl\to k^\cro$, i.e.
$A$ is spanned by $e_x, x\in\gl$ with multiplication 
\begin{equation}\label{skg}
e_xe_y = \alpha(x,y)e_{x+y}.
\end{equation} 
\newline
The $\h$-action has a form: 
\begin{equation}\label{ac}
y(e_x) = (x,y)e_x,\quad \forall x\in\gl, y\in\h.
\end{equation}
\end{prop}
\begin{proof}
We are going to show first that the presence of a non-degenerate multiplicative bilinear form together with the condition $A^\h=k$ force homogeneous elements with respect to the grading (\ref{algg}) to be invertible. Indeed, multiplicativity of a non-degenerate form implies that it is compatible with the grading, in particular for any $x\in\gl$ the restriction $\beta\se A_x\otimes A_{-x}\to k$ is non-degenerate. From the other side, multiplicativity of $\beta$ and the condition $A^\h=k$ imply that for $a\in A_x, b\in A_{-x}$ 
$$ab = \beta(ab,1)1 = \beta(a,b)1.$$
By non-degeneracy of $\beta$ for any non-zero $a$ there is $b$ such that $\beta(a,b)=1$. Hence $b=a^{-1}$. 
Now we can show that generalised eigenspaces $A_x$ are at most one-dimesional. Indeed for non-zero $a,b\in A_x$ we have that $b^{-1}a = \lambda 1$ for some $\lambda\in k$. Thus $a=\lambda b$. 

By choosing non-zero elements $e_x\in A_x$ for $x\in\gl$ we see that the multiplication should have a form (\ref{skg}) for some non-zero $\alpha:\gl\times\gl\to k^\cro$. Associativity of the multiplication implies that $\alpha$ is a normalised 2-cocycle.

Finally, being unique (up to a scalar) generalized eigenvector with a given character, $e_x$ has to be a genuine eigenvector, i.e. the $\h$-action on it has to have a form (\ref{ac}). 
\end{proof}

An algebra $A$ in $\cC(\h,\Omega)$ is {\em commutative} if 
\begin{equation}\label{com}
\mu(e^{\pi i\Omega}(b\otimes a)) = ab,\quad \forall a,b\in A.
\end{equation}

First we calculate the effect of Casimir on tensor product of eigenvectors.
\begin{lem}\label{cas}
Let $a,b\in A$ are such that 
$$z(a) = (x,z)a,\quad z(b) = (y,z)b,\quad \forall z\in\h.$$
Then
$$\Omega(b\otimes a) = (x,y)b\otimes a,\quad \omega(a) = (x,x)a.$$
\end{lem}
\begin{proof}
Writing $\Omega = \sum_i x_i\otimes x_i$ in the orthonormal basis for the form $(x_i,x_j)=\delta_{i,j}$ we get
$$\Omega(b\otimes a) = \sum_i x_i(b)\otimes x_i(a) = \sum_i(y,x_i)(x,x_i)(b\otimes a) = (x,y)(b\otimes a)$$
and
$$\omega(a) = \sum_i x_ix_i(a) = \sum_i(x,x_i)(x,x_i)a = (x,x)a.$$
\end{proof}

It is well known that up to an isomorphism the skew group algebra $k[\gl,\alpha]$ depend only on the cohomology class of the cocycle $\alpha$. 

Here we assume that our ground field $k$ is the field of complex numbers $\C$.
\begin{theo}\label{coa}
A commutative algebra $A$ in the braided category $\cC(\h,\Omega)$, which has a non-degenerate multiplicative bilinear form and has trivial subalgebra of invariants $A^\h=k$, has a form  $k[\gl,\alpha]$, where $\gl$ is a subgroup of $\h$ such that the restriction of the form on $\gl$ is integer and even 
$$(x,y)\in\Z,\quad (x,x)\in 2\Z,\quad x,y\in\gl.$$
The cohomology class of the cocycle $\alpha$ (and hence the isomorphism class of $k[\gl,\alpha]$) is uniquely defined by the condition:
\begin{equation}\label{cc}
\frac{\alpha(x,y)}{\alpha(y,x)} = e^{\pi i(x,y)}.
\end{equation}
\end{theo}
\begin{proof}
Putting $a=e_x$ and $b=e_y$ in the equation (\ref{com}) and using lemma \ref{cas} we get
$$\alpha(x,y)e_{x+y} = e_xe_y = e^{\pi i(x,y)}e_ye_x = \alpha(y,x)e^{\pi i(x,y)}e_{x+y},$$ which gives equation (\ref{cc}). 
By setting $x=y$ we have that $e^{\pi i(x,x)}=1$ for all $x\in\gl$, which means that $(x,x)$ must be an even integer. This implies that $(x,y)$ is an integer for any $x,y\in\gl$ since 
$$(x,y) = \frac{1}{2}((x+y,x+y) - (x,x) - (y,y)).$$
By the exact sequence of universal coefficients:
$$\xymatrix{0 \ar[r] & Ext(\gl,k^\cro) \ar[r] & H^2(\gl,k^\cro) \ar[r] & Hom(\Lambda^2_\Z\gl,k^\cro)\ar[r] & 0}$$
The group $k^\cro=\C^\cro=\C/\Z$ is divisible, hence $Ext(\gl,k^\cro)=0$ (see \cite{ce}) and the cohomology class of $\alpha$ is uniquely defined by its skew-symmetrisation. 
\end{proof}

\begin{corr}
Isomorphism classes of commutative algebras $A$ in the braided category $\cC(\h,\Omega)$, which have a non-degenerate multiplicative bilinear form and have trivial subalgebra of invariants $A^\h=k$, are in one to one correspondence with subgroups $\gl\subset\h$ such that the restriction of the form on $\gl$ is integer and even 
\begin{equation}\label{ii}
(x,y)\in\Z,\quad (x,x)\in 2\Z,\quad x,y\in\gl.
\end{equation}
\end{corr}

\section{Local modules in Drinfeld categories of abelian Lie algebras}\label{locmod}

Let $A$ be a commutative algebra in $\cC(\h,\Omega)$. A left $A$-module in $\cC(\h,\Omega)$ is {\em local} if 
\begin{equation}\label{loc}
\mu(e^{2\pi i\Omega}(a\otimes m)) = am,\quad \forall a\in A, m\in M.
\end{equation}

The following is a generalization of lemma \ref{cas}.
\begin{lem}\label{casm}
Let $A$ be an algebra and $M$ be a left $A$-module in $\cC(\h,\Omega)$.
Let $a\in A$ is such that 
$z(a) = (x,z)a,$ for all $z\in\h.$
Then
$$\Omega(a\otimes m) = a\otimes x(m),\quad \forall m\in M.$$
\end{lem}
\begin{proof}
Writing $\Omega = \sum_i x_i\otimes x_i$ in the orthonormal basis for the form $(x_i,x_j)=\delta_{i,j}$ we get
$$\Omega(a\otimes m) = \sum_i x_i(a)\otimes x_i(m) =a\otimes(\sum_i(x,x_i)x_i(a)) = a\otimes x(m).$$
\end{proof}

Let $\gl\subset\h$ and $\alpha$ be as in theorem \ref{coa}.  
\begin{lem}\label{lmd}
A module $M$ over the commutative algebra $k[\gl,\alpha]$ in the braided category $\cC(\h,\Omega)$ is local if and only if $M$ is semisimple as an $\gl$-module: 
\begin{equation}\label{lm}
M = \bigoplus_{\chi\in\gl^\vee}M_\chi,\quad M_\chi = \{m\in M|\ x(m) = \chi(x)m,\ \forall x\in\gl\}
\end{equation}
and for $M_\chi\not=0$ the character $\chi\se\gl\to k$ has integer values: $\chi(\gl)\subset\Z$. 
\newline
The algebra action permutes the components as follows:
$$e_xM_\chi = M_{\chi+(x,\da)}.$$
\end{lem}
\begin{proof}
Putting $a=e_x$ in the equation (\ref{loc}) and using lemma \ref{casm} we get
$e_xm = e_xe^{2\pi i x}(m),$ which means that for $x\in\gl$ the operator $e^{2\pi i x}$ is the identity on $M$. Thus $M$ is semisimple as a $\gl$-module with the characters as in (\ref{lm}). 

For $m\in M_\chi$ and $y\in\gl$ we have
$$y(e_xm) = y(e_x)m + e_xy(m) = (x,y)e_xm + \chi(y)e_xm.$$
\end{proof}
We define the {\em support} of a local $k[\gl,\alpha]$-module $M$ as
$$supp(M) = \{ \chi\in Hom(\gl,\Z)|\ M_\chi\not= 0\}.$$ By lemma \ref{lmd} the support of a local $A$-module is a union of $\gl$-orbits in $Hom(\gl,\Z)$ with respect to $\gl$-action on $Hom(\gl,\Z)$ given by the homomorphism
$$\gl\to Hom(\gl,\Z),\quad x\mapsto (x,\da).$$
Now we are going to present the group of $\gl$-orbits in $Hom(\gl,\Z)$ in a slightly dif and only iferent way. Define
$$\gl^\di = \{ x\in\h |\ (x,\gl)\subset \Z\}.$$
The commutative diagram
$$\xymatrix{
\h \ar@{=}[r] & \h^* \ar@{->>}[r] & Hom(\gl,k)\\
\gl^\di \ar@^{(->}[u] \ar[rr] && Hom(\gl,\Z) \ar@^{(->}[u]
}$$
implies that the homomorphism $\gl^\di\to Hom(\gl,\Z)$ is surjective. Clearly its kernel is $\gl^\ort = \{ x\in\h |\ (x,\gl)=0\}$. Thus $Hom(\gl,\Z)$ can be identified with $\gl^\di/\gl^\ort$. If the form has integer values on $\gl$, the group $\gl$ embeds in $\gl^\di$ and the cokernel of the map $\gl\to Hom(\gl,\Z)$ is $ \gl^\di/(\gl+\gl^\ort)$.

\begin{prop}\label{grad} 
Let $A$ be $k[\gl,\alpha]$. 
The category ${_{A}}{\cC(\h,\Omega)^{loc}}$ has a monoidal grading by the group $\gl^\di/(\gl+\gl^\ort)$ :
\begin{equation}\label{cgr}
{_{A}}{\cC(\h,\Omega)^{loc}} = \bigoplus_{X\in \gl^\di/(\gl+\gl^\ort)}{_{A}}{\cC(\h,\Omega)_X^{loc}},
\end{equation}
where the category ${_{A}}{\cC(\h,\Omega)_X^{loc}}$ is the full subcategory of ${_{A}}{\cC(\h,\Omega)^{loc}}$, consisting of modules $M$ with support $X$ (or rather the corresponding $\gl$-orbit in $Hom(\gl,\Z)$).
\end{prop}
\begin{proof}
Clearly the support of an indecomposable local $k[\gl,\alpha]$-module is an orbit, which provides the decomposition (\ref{cgr}). 
To establish monoidality of the grading we need to show that the tensor product $M\otimes_AN$ of local modules $M\in {_{A}}{\cC(\h,\Omega)_X^{loc}}$, $N\in {_{A}}{\cC(\h,\Omega)_Y^{loc}}$ belongs to ${_{A}}{\cC(\h,\Omega)_{X+Y}^{loc}}$. This follows from the following inclusion of eigenspaces $M_\chi\otimes N_\xi\subset (M\otimes N)_{\chi+\xi}$. 
\end{proof}

Note that the kernel of the restriction $(\da, \da)|_{\gl^\ort}$ coincides with $\gl^\ort\cap\gl^{\ort\ort}$. Note also that 
the double orthogonal $\gl^{\ort\ort}$ coincides with the vector subspace of $\h$ generated by $\gl$. Thus the induced form on $\overline\gl = \gl^\ort/(\gl^\ort\cap\gl^{\ort\ort})$ is non-degenerate. Let $\overline\Omega$ be the Casimir 
of this form on $\overline\gl$. 

Now we are ready to characterise the degree 0 component of the category of local modules. 
\begin{prop}\label{dz}
The category ${_{A}}{\cC(\h,\Omega)_0^{loc}}$ is equivalent, as a braided monoidal category, to the Drinfeld category $\cC(\overline\gl,\overline\Omega)$.
\end{prop}
\begin{proof}
We will prove the proposition by constructing two braided equivalences:
$${_{A}}{\cC(\h,\Omega)_0^{loc}}\to{_{A_0}}{\cC(\h/\gl^{\ort\ort},\Omega')}\to \cC(\overline\gl,\overline\Omega).$$
First we will prove that the category ${_{A}}{\cC(\h,\Omega)_0^{loc}}$ is equivalent, as a braided monoidal category, to the category ${_{A_0}}{\cC(\h/\gl^{\ort\ort},\Omega')}$ of modules over the group algebra $A_0=k[\gl\cap\gl^\ort]$ in the Drinfeld category $\cC(\h/\gl^{\ort\ort},\Omega')$. Here $\Omega' = (f\otimes f)(\Omega) $ is the image of the Casimir $\Omega\in\h^{\otimes 2}$ under the epimorphism $f:\h\to\h/\gl^{\ort\ort}$.
  
By the definition of the subcategory ${_{A}}{\cC(\h,\Omega)_0^{loc}}$ a local $A$-module $M$ belongs to ${_{A}}{\cC(\h,\Omega)_0^{loc}}$ if and only if its eigenspace decomposition, as an $\gl$-module, has a form 
\begin{equation}\label{tdd}
M = \bigoplus_{x\in\gl^\ort}M_x,\quad M_x=\{m\in M|\ y(m)=(x,y)m\quad \forall y\in\gl\}.
\end{equation}
Note that the $A$-action amounts to: $e_xM_y = M_{x+y}$. 

Define a functor ${_{A}}{\cC(\h,\Omega)_0^{loc}}\to{_{A_0}}{\cC(\h/\gl^{\ort\ort},\Omega')}$ by an assignment $M\mapsto M_0$, where $M_0 = \{m\in M|\ z(m) = 0\ \forall z\in \gl\} = M^\gl$ is the space of $\gl$-invariants (which coincides with the space of $\gl^{\ort\ort}$-invariants). The $A$-action on $M$ gives a $A_0$-action on $M_0$. Note that $A_0$ coincides with the subalgebra of $A$, spanned by $e_t$ with $t\in \gl\cap\gl^\ort$. Since the form is zero on $\gl\cap\gl^\ort$, $A_0$ is isomorphic to the (untwisted) group algebra $k[\gl\cap\gl^\ort]$. 

A quasi-inverse functor ${_{A_0}}{\cC(\h/\gl^{\ort\ort},\Omega')}\to {_{A}}{\cC(\h,\Omega)_0^{loc}}$ sends $N$ into $A\otimes_{A_0} N$ with the $\h$-action:
$z(a\otimes n) = z(a)\otimes n + a\otimes z(n)$ and the $A$-action: $a(b\otimes n) = ab\otimes n$. The $A$-module $A\otimes_{A_0} N$ is clearly local, since the $\gl$-action on $N$ is trivial. The monoidal structure of this functor is given by
$$(A\otimes_{A_0} N)\otimes_A(A\otimes_{A_0} L)\to A\otimes_{A_0}(N\otimes L),\quad (a\otimes n)\otimes_A(b\otimes l)\mapsto ab\otimes n\otimes l.$$
\newline
Since $N^\gl=0$ and $A^\gl=k$, the natural inclusion $N\to (A\otimes_{A_0}N)^\gl$ (induced by the unit element of $A$) is an isomorphism. 
The map $A\otimes_{A_0} M^\gl\to M$ (induced by the $A$-action) is an isomorphism because of the special shapes of the eigenspace decomposition (\ref{tdd}) and the $A$-action. Indeed it follows from the comparison of decompositions for $A\otimes_{A_0} M^\gl$ and $M$:
$$(A\otimes_{A_0} M^\gl)_x = e_x\otimes M^\gl \to e_xM^\gl = M_x.$$
The fact, that the functor ${_{A}}{\cC(\h,\Omega)_0^{loc}}\to{_{A_0}}{\cC(\h/\gl^{\ort\ort},\Omega')}$ is braided, can be checked directly. Indeed it transforms the braiding $c_{M,N}(m\otimes m) = e^{\pi i\Omega}(n\otimes m)$ in ${_{A}}{\cC(\h,\Omega)_0^{loc}}$ into 
$$\overline c_{M^\gl,N^\gl}(m\otimes n) = e^{\pi i(f\otimes f)(\Omega)}(n\otimes m),$$ where as before $f\se\h\to\h/\gl^{\ort\ort}$ is the quotient map. 

Now we construct a functor 
\begin{equation}\label{equ}
{_{A_0}}{\cC(\h/\gl^{\ort\ort},\Omega')}\to \cC(\overline\gl,\overline\Omega)
\end{equation}
by sending a $k[\gl\cap\gl^\ort]$-module $N$ into the space $N_{\gl\cap\gl^\ort} = N/\sum_{t\in\gl\cap\gl^\ort}(1-e_t)(N)$ of coinvariants. Since $\gl^\ort$ acts trivially on $A_0$, the $\gl^\ort/(\gl^\ort\cap\gl^{\ort\ort})$-action on $N$ descents to $N_{\gl\cap\gl^\ort}$. The natural map $(N\otimes_{A_0}P)_{\gl\cap\gl^{\ort}}\to N_{\gl\cap\gl^{\ort}}\otimes P_{\gl\cap\gl^{\ort}}$ defines a monoidal structure.

Rather than constructing a quasi-inverse to the functor (\ref{equ}) (which would involve a delicate choice of sections for certain maps) we will prove that this functor is fully faithful and surjective on objects. To prove that the functor is surjective on objects, choose a section $\sigma\se\h/\gl^{\ort\ort}\to\overline\gl$ of the embedding of abelian Lie algebras $\overline\gl=(\gl^{\ort\ort}+\gl^\ort)/\gl^{\ort\ort}\to\h/\gl^{\ort\ort}$. This will allow us to extend an $\overline\gl$-module structure on $Q$ to a $\h/\gl^{\ort\ort}$-module structure. Then the free $A_0$-module $A_0\otimes Q$ (with the diagonal $\h/\gl^{\ort\ort}$-action) will have a property $(A_0\otimes Q)_{\gl\cap\gl^\ort} = Q$. This establishes surjectivity on objects as well as fullness. 
To show that the functor (\ref{equ}) is faithful we need to prove that if for a $A_0$- and $\h/\gl^{\ort\ort}$-linear map $f:N\to P$ between $A_0$-modules in $\cC(\h/\gl^{\ort\ort},\Omega')$ 
the induced map of coinvariants $f_{\gl\cap\gl^\ort}\se N_{\gl\cap\gl^\ort}\to P_{\gl\cap\gl^\ort}$ is zero then $f$ is zero. In other words if the image of $f$ is in $\sum_{t\in\gl\cap\gl^\ort}(1-e_t)(P)$ then $f$ is zero. The last follows from lemma \ref{sub}, since $im(f)$ is an $\h$-module.
\end{proof}

\begin{lem}\label{sub}
Let $P$ be an $A_0$-modules in $\cC(\h/\gl^{\ort\ort},\Omega')$. Then the only $\h$-module inside $\sum_{t\in\gl\cap\gl^\ort}(1-e_t)(P)$ is the zero subspace.
\end{lem}
\begin{proof}
Let $J\subset \sum_{t\in\gl\cap\gl^\ort}(1-e_t)(P)$ be an $\h$-module. Assume that $J$ is non zero and write an element of $J$ as $\sum_{t\in\gl\cap\gl^\ort}(1-e_t)(p_t)$ with (at least one of) $p_t\not\in\sum_{t\in\gl\cap\gl^\ort}(1-e_t)(P)$. The identity (with $z\in\h$)
$$z(\sum_{t\in\gl\cap\gl^\ort}(1-e_t)(p_t)) - \sum_{t\in\gl\cap\gl^\ort}(1-e_t)(z(p_t)) = \sum_{t\in\gl\cap\gl^\ort}(t,z)e_t(p_t)$$ shows that the right hand side is in $J$. By varying $z$ (and solving a linear system) we can see that $e_t(p_t)$ belongs to $J$ for all $t$, which is in contradiction with our assumption. Hence $J=0$.
\end{proof}

Define a bimultiplicative function $c\se\gl^\di\times\gl^\di\to k^\cro$ by $c(x,y) = e^{\pi i(x,y)}$. Let $\cG(\gl^\di)$  be the category of finite dimensional $\gl^\di$-graded vector spaces. Define a brading on $\cG(\gl^\di)$ by means of $c$ and denote this braided monoidal category by $\cG(\gl^\di,c)$   (cf. Appendix). 

\begin{prop}\label{fun}
There is a braided monoidal functor $\cG(\gl^\di,c)\to{_{A}}{\cC(\h,\Omega)^{loc}}$ such that the image of $[x]$ belongs to 
${_{A}}{\cC(\h,\Omega)_X^{loc}}$, where $X$ is the coset of $x$. 
\end{prop}
\begin{proof}
For $x\in\h$ define an $\h$-module $A(x)$ as a span of $\{e^x_y,\ y\in \gl\}$ with the $\h$-action: 
$$z(e^x_y) = (z,x+y)e^x_y.$$ 
Define an $A$-action on $A(x)$ by 
$e_ue^x_y = \alpha(u,y)e^x_{u+y}$, which turns $A(x)$ into a $A$-module in $\cC(\h,\Omega)$. Indeed, 
$$z(e_u)e^x_y + e_uz(e^x_y) = (u,z)e_ue^x_y + (x+y,z)e_ue^x_y$$
coincides with
$$z(e_ue^x_y) = \alpha(u,y)z(e^u_{u+y}) = \alpha(u,y)(u+x+y,z)e^x_{u+y}$$
Clearly the module $A(x)$ has the following eigenspace decomposition $A(x) = \oplus_{y\in\gl}A(x)_{(x+y,\da)}$. 
In particular, for $x\in\gl^\di$ the $A$-module $A(x)$ is local and belongs to the subcategory ${_{A}}{\cC(\h,\Omega)^{loc}_X}$ where $X$ is the $\gl+\gl^\ort$-coset of $x$. 
\newline
Now we need to calculate tensor products $A(x)\otimes_AA(y)$ for $x,y\in\h$. The map
$$\phi_{x,y}:A(x)\otimes A(y)\to A(x+y),\quad e^x_u\otimes e^y_v\mapsto \alpha(u,v)e^{-\pi i(x,v)}e^{x+y}_{u+v}$$
is obviously  $\h$-linear. It has a property 
$$\phi_{x,y}(e_we^x_u\otimes e^y_v) = \phi_{x,y}(e^{\pi i(w,x+u)}(e^x_u\otimes e_we^y_v)).$$ 
Indeed, by 2-cocycle property of $\alpha$ and since $e^{\pi i(u,w)} = \alpha(u,w)\alpha(w,u)^{-1}$
$$\phi_{x,y}(e_we^x_u\otimes e^y_v) = \alpha(w,u)\phi_{x,y}(e^x_{w+u}\otimes e^y_v) = \alpha(w,u)\alpha(w+u,v)e^{-\pi i(x,v)}e^{x+y}_{w+u+v}$$
coincides with
$$\phi_{x,y}(e^{\pi i(w,x+u)}(e^x_u\otimes e_we^y_v) = e^{\pi i(w,x+u)}\phi_{x,y}(e^x_u\otimes \alpha(w,v)e^y_{w+v}) = $$ $$\alpha(w,v)\alpha(u,w+v)e^{\pi i(w,x+u)}e^{-\pi i(x,w+v)}e^{x+y}_{w+u+v}.$$
Thus the map $\phi_{x,y}$ factors through the map $A(x)\otimes_AA(y)\to A(x+y)$. Using the relation $e_we^x_u\otimes e^y_v = e^{\pi i(w,x+u)}(e^x_u\otimes e_we^y_v)$, valid in $A(x)\otimes_AA(y)$, we can see that the map $\phi_{x,y}: A(x)\otimes_AA(y)\to A(x+y)$ is an isomorphism. For $x,y,z\in\h$ 
maps $\phi$ fit into a commutative diagram
$$\xymatrix{
A(x)\otimes_AA(y)\otimes_AA(z) \ar[r]^{\phi_{x,y}1} \ar[d]_{1\phi_{y,z}} & A(x+y)\otimes_AA(z) \ar[d]^{\phi_{x+y,z}} \\
A(x)\otimes_AA(y+z) \ar[r]^{\phi_{x,y+z}} & A(x+y+z)
}$$ Indeed, the top-right composition acts as
$$e^x_u\otimes e^y_v\otimes e^z_w\mapsto \alpha(u,v)e^{-\pi i(x,v)}e^{x+y}_{u+v}\otimes e^z_w\mapsto$$
$$\alpha(u,v)\alpha(u+v,w)e^{-\pi i((x,v)+(x+y,w))}e^{x+y+z}_{u+v+w},$$ while the left-bottom composite has a form
$$e^x_u\otimes e^y_v\otimes e^z_w\mapsto \alpha(v,w)e^{\pi i(y,w)}e^x_u\otimes e^{y+z}_{u+w} \mapsto$$
$$\alpha(v,w)\alpha(u,v+w)e^{\pi i((y,w)+(x,v+w))}e^{x+y+z}_{u+v+w}.$$ Hence we have a monoidal functor $\cG(\gl^\di)\to {_{A}}{\cC(\h,\Omega)^{loc}}$, which is in fact a braided functor.  Indeed, the diagram
$$\xymatrix{
A(x)\otimes_AA(y) \ar[r]^{\phi_{x,y}} \ar[d]_{c_{A(x),A(y)}} & A(x+y) \ar[d]^{c(x,y)1}\\
A(y)\otimes_AA(x) \ar[r]^{\phi_{y,x}} & A(x+y)
}$$ commutes:
\newline
the top-right composition acts as
$$e^x_u\otimes e^y_v\mapsto \alpha(u,v)e^{\pi i(x,v)}e^{x+y}_{u+v}\mapsto \alpha(u,v)e^{\pi i((x,v)+(x,y))}e^{x+y}_{u+v},$$
which coincides with the action of the left-bottom composite
$$e^x_u\otimes e^y_v\mapsto e^{\pi i(x+u,y+v)}e^y_v\otimes e^x_u\mapsto \alpha(v,u)e^{\pi i((x+u,y+v)+(y,v))}e^{x+y}_{u+v}.$$
This gives us the desired braided monoidal functor $[x]\mapsto A(x)$. 
\end{proof}

\begin{rem}\label{isomo}
\end{rem}
Note that for $x\in\gl$ the local $A$-module $A(x)$, defined in the proof of proposition \ref{fun}, is isomorphic to $A$. Indeed define a map $\psi_x\se A(x)\to A$ by $e^x_u\mapsto \alpha(u,x)e_{x+u}$. While $\h$-linearity of $\psi_x$ is obvious, $A$-linearity follows from 2-cocycle property of $\alpha$, which implies that
$$\psi_x(e_ve^x_u) = \psi_x(\alpha(v,u)e^x_{v+u}) = \alpha(v,u)\alpha(v+u,x)e_{x+v+u}$$ coincides with 
$$e_v\psi_x(e^x_u) = e_v\alpha(u,x)e_{x+u} = \alpha(u,x)\alpha(v,u+x)e_{x+v+u}.$$
Moreover, for $x,y\in\gl$ the diagram
$$\xymatrix{
A(x)\otimes_AA(y) \ar[r]^{\phi_{x,y}} \ar[d]_{\psi_x\psi_y} & A(x+y) \ar[d]^{\psi_{x+y}}\\
A\otimes_AA \ar[r]^\mu & A
}$$ 
commutes up to multiplication by $\alpha(x,y)$. Indeed the top-right composition has the effect
$$e^x_u\otimes_Ae^y_v\mapsto \alpha(u,v)e^{\pi i(x,v)}e^{x+y}_{u+v}\mapsto \alpha(u,v)e^{\pi i(x,v)}\alpha(u+v,x+y)e^{x+y+u+v},$$
while the left-bottom composition acts as
$$e^x_u\otimes_Ae^y_v\mapsto \alpha(u,x)\alpha(v,y)e_{x+u}\otimes_Ae_{y+v}\mapsto  \alpha(u,x)\alpha(v,y)\alpha(x+u,y+v)e_{x+u+y+v}.$$
The ratio of the coefficients
$$\alpha(u,x)\alpha(v,y)\alpha(x+u,y+v)\alpha(u,v)^{-1}\alpha(u+v,x+y)^{-1}e^{\pi i(x,v)}$$
is equal to 
$$\alpha(x,y)e^{\pi i((v,y)+(x+y,v)+(x,v))}d(\alpha)(u,x,y+v)d(\alpha)(x,y,v)^{-1}d(\alpha)(u,v,x+y)^{-1},$$
where 
$$d(\alpha)(x,y,z) = \alpha(x,y)\alpha(x+y,z)\alpha(y,z)^{-1}\alpha(x,y+z)^{-1}$$ equals 1 by the 2-cocycle property of $\alpha$.

\section{Finite categories of local modules}\label{finloc}

Here we describe commutative algebras $A=k[\gl,\alpha]$ which have only finite number of simple local modules.
By proposition \ref{dz} the category of local modules ${_{A}}{\cC(\h,\Omega)^{loc}}$ contains the Drinfeld category $\cC(\overline\gl,\overline\Omega)$, where $\overline\gl = \gl^\ort/(\gl^\ort\cap\gl^{\ort\ort})$. For a non-zero $\overline\gl$ the category $\cC(\overline\gl,\overline\Omega)$ has a continues family of simple objects. Thus for $A$, to have only finite number of simple local modules we need to assume that $\overline\gl =0$, which motivates the following definition. 

We call a subgroup $\gl\subset\h$ {\em coisotropic} if $\gl^\ort\subset\gl^{\ort\ort}$. 

Now we show that for a coisotropic $\gl$ the category of local $A$-modules is equivalent to the category of vector spaces graded by $\gl^\di/(\gl+\gl^\ort)$. To formulate the result we need to describe the quadratic function on $\gl^\di/(\gl+\gl^\ort)$ controlling the associativity and braiding constraints.

We choose a section $\sigma\se\gl^\di/(\gl+\gl^\ort)\to\gl^\di$ and define a function $q\se\gl^\di/(\gl+\gl^\ort)\to k^\cro$ by $q(X) = e^{\pi i(\sigma(X),\sigma(X))}$. Note that $q$ does not depend on the choice of the section. 
\begin{prop}
Let $\gl\subset\h$ be a coisotropic subgroup and $\alpha$ be a 2-cocycle of $\gl$ satisfying  the conditions of theorem \ref{coa}. Then the category ${_{A}}{\cC(\h,\Omega)^{loc}}$ of local modules over the commutative algebra $A=k[\gl,\alpha]$ is equivalent, as a braided monoidal category, to the category $\cG(\gl^\di/(\gl+\gl^\ort),q)$ of $\gl^\di/(\gl+\gl^\ort)$-graded vector spaces, with the associativity and the braiding defined by $q$. 
\end{prop}
\begin{proof}
By proposition \ref{grad} the category ${_{A}}{\cC(\h,\Omega)^{loc}}$ is a braided monoidal category, graded by the group $\gl^\di/(\gl+\gl^\ort)$. 
By proposition \ref{dz}, for coisotropic $\gl$, the degree 0 component is trivial, i.e. ${_{A}}{\cC(\h,\Omega)^{loc}_0}$ is equivalent to the category of vector spaces $\Vect$. 
By proposition \ref{fun} we have a braided monoidal functor  $\cG(\gl^\di,c)\to{_{A}}{\cC(\h,\Omega)^{loc}}$, compatible with the grading.
Finally isomorphisms from remark \ref{isomo} allow us to apply theorem \ref{grgr} from the appendix.
\end{proof}

Note that the group $Hom(\gl,\Z)$ is torsion free. Since the vector space $Hom(\gl,\Z)\otimes_\Z k = Hom(\gl,k)\subset\h$ is finite dimensional,  the group $Hom(\gl,\Z)$ is a free abelian of finite rank. Hence the group $\gl^\di/(\gl+\gl^\ort)$ is finitely generated and as such is a sum of a finite abelian and a free abelian group of finite rank. The rank coincides with the dimension of the vector space $\gl^\di/(\gl+\gl^\ort)\otimes_\Z k$. We can identify the vector space $\gl^\di/(\gl+\gl^\ort)\otimes_\Z k$ with the cokernel of the map 
\begin{equation}\label{dua}
\gl^{\ort\ort}\to Hom(\gl,k),\quad x\mapsto (x,\da).
\end{equation} 
Note the dimension of the cokernel of (\ref{dua}) equals the dimension of its kernel $\gl^{\ort\ort}\cap\gl^\ort$. Thus the group $\gl^\di/(\gl+\gl^\ort)$ is finite if and only if $\gl^{\ort\ort}\cap\gl^\ort=0$. For a coisotropic $\gl$ it means that $\gl^\ort=0$. In particular $\gl$, as a subgroup of $Hom(\gl,\Z)$, is a free abelian group. The short exact sequence
$$0\to \gl^\ort\to \h\to Hom(\gl,k)\to 0$$ shows that in coisotropic case $Hom(\gl,k)=\h$, i.e. $\gl$ is a lattice in $\h$. 

Thus we have the following.
\begin{corr}
The commutative algebra $A=k[\gl,\alpha]$ has (up to isomorphism) a finite number of simple local modules if and only if $\gl\subset\h$ is an even lattice. In that case the category ${_{A}}{\cC(\h,\Omega)^{loc}}$ is equivalent, as a braided monoidal category, to the category $\cG(\gl^\di/\gl,q)$ of $\gl^\di/\gl$-graded vector spaces, with the associativity and the braiding defined by $q$. 
\end{corr}

\section*{Appendix. Graded monoidal categories}

Here we recall basic facts about gradings on monoidal categories in general and braided monoidal categories of group-graded vector spaces in particular. 

Let $H$ be a group. An {\em $H$-grading} on a monoidal category $\cC$ is a decomposition $\cC = \oplus_{h\in H}\cC_h$ such that for $X\in\cC_f$, $Y\in\cC_g$ the tensor product $X\otimes Y$ belongs to $\cC_{fg}$. 

Simplest examples of graded monoidal categories are provided by the following construction. 
For a group $H$ denote by $\cG(H)$ the category of finite dimensional $H$-graded vector spaces. Tensor product of $H$-graded vector spaces can be equipped with the $H$-grading 
$$(V\otimes U)_h = \bigoplus_{fg=h}V_f\otimes U_g.$$ This makes the category $\cG(H)$ monoidal. Clearly it is $H$-graded with $V_f$ belonging to $\cG(H)_f\simeq\Vect$. We will also be interested in braidings on the categories $\cG(H)$.

Now let $H$ be an abelian group and $c:H\times H\to k^\cro$ be a bi-multiplicative function (multiplicative in each variable). It is straightforward to see that $$c_{V,U}(v\otimes u) = c(f,g)(u\otimes v),\quad v\in V_f, u\in U_g$$ defines a braiding on $\cG(H)$. We will denote by $\cG(H,c)$ the corresponding braided monoidal category. 

Note that $\cG(H,c)$ is not the most general braided monoidal category structure on the category of graded vector spaces. In full generality such structures were classified in \cite{js}. Here we formulate the result. A general solution for the associativity constraint for the tensor product of $H$-graded vector spaces is given by a (normalised) 3-cocycle $a:H\times H\times H\to k^\cro$:
$$a_{V,U,W}(v\otimes (u\otimes w)) = a(f,g,h)(v\otimes u)\otimes w,\quad v\in V_f, u\in U_g, w\in W_h.$$
A braiding, compatible with the associativity given by $a$ corresponds to a function $c:H\times H\to k^\cro$:
$$c_{V,U}(v\otimes u) = c(f,g)u\otimes v,\quad v\in V_f, u\in U_g.$$
Hexagon coherence axioms for the braiding $c$ are equivalent to the equations:
$$a(g,h,f)c(f,gh)a(f,g,h) = c(f,h)a(g,f,h)c(f,g),$$
$$a(h,f,g)^{-1}c(fg,h)a(f,g,h)^{-1} = c(f,h)a(f,h,g)^{-1}c(g,h).$$
The pair $(a,c)$ is called an {\em abelian} 3-cocycle. 
Up to braided monoidal equivalence the braided monoidal category corresponding to $(a,c)$ depends only on the (abelian) cohomology class of $(a,c)$, which is determined by the quadratic function $q(f) = c(f,f)$. Recall (e.g. from \cite{js}) that a function $q:H\to k^\cro$ is quadratic if and only if $q(f^{-1})=q(f)$ for all $f\in H$ and the function $\sigma:H\times H\to k^\cro$
$$\sigma(f,g) = g(fg)q(f)^{-1}q(g)^{-1}$$ is multiplicative in each argument (a so-called {\em bi-character}). 
In other words braided equivalence classes of braided monoidal structures on the category of $H$-graded vector spaces correspond to quadratic functions on $H$. We will denote by $\cG(H,q)$ a representative of the class corresponding to the quadratic function $q$. 

Now we are ready to study the situation we have in section \ref{locmod}. 
Let $F\se\cG(H,c)\to\cC$ be a braided monoidal functor. Let $K\subset H$ be a subgroup and for each $x\in K$ let $\psi_x\se F([x])\to I$ be an isomorphism (we will assume that $\psi_0$ is the identity or rather unit isomorphism of the functor $F$). Being an automorphism of the unit object of $\cC$ the (counterclockwise) composition of arrows of the diagram
$$\xymatrix{ F([x]\otimes[y]) \ar@{=}[d] \ar[r]^(.45){F_{[x],[y]}} & F([x])\otimes F([y]) \ar[d]^{\psi_x\psi_y}\\ F([x+y]) \ar[r]^{\psi_{x+y}} & I}$$
is a multiple of the identity. Let $\alpha(x,y)$ be the coefficient. Then we have the following properties for $\alpha:K\times K\to k^\cro$.
\begin{lem}
The function $\alpha$ is a normalised 2-cocycle such that $$\alpha(x,y)\alpha(y,x)^{-1} = c(x,y),\quad \forall x,y\in K.$$
\end{lem}
\begin{proof}
The 2-cocycle property follows from the commutativity of the diagram:
$$\xygraph{ *+{F([x][y][z])}
(
:[u(3.7)r(4.6)] *+{F([x][y])F([z])} ^{F_{[x][y],[z]}}
 (
 :@{=}[d(1.2)l(.2)] *+{F([x+y])F([z])}
 :[d(.9)l(.13)]="u" *+{I} ^{\psi_{x+y}\psi_z}
 :[r(1.8)d(2.3)]="e" *+{I} ^{\alpha(x,y)1}
 ,
 :[r(3.7)d(4.6)]="l" *+{F([x])F([y])F([z])} ^{F_{[x],[y]}1}
 :"e" *+{I} _(.7){\psi_x\psi_y\psi_z}
 )
,
:@{=}[d(.2)r(1.9)] *+{F([x+y+z])}
:[r(1.7)d(.2)] *+{I} ^(.7){\psi_{x+y+z}}
 (
 :[u(1.8)r(.645)] ^{\alpha(x+y,z)1}
 ,
 :[d(2)]="d" *+{I} _{\alpha(x,y+z)1}
 :[r(2.3)u(1.5)] _{\alpha(y,z)1}
 )
,
:[r(3.7)d(4.6)] *+{F([x])F([y+z])} _{F_{[x][y][z]}}
 (
 :@{=}[u(1.3)r(.1)] *+{F([x])F([y+z])}
 :[u(.7)l(.15)] _{\psi_x\psi_{y+z}}
 ,
 :[r(4.4)u(3.5)] _{1F_{[y],[z]}}
 )
)
}$$
The property $\alpha(x,y)\alpha(y,x)^{-1} = c(x,y)$ follows from the commutativity of the diagram:
$$
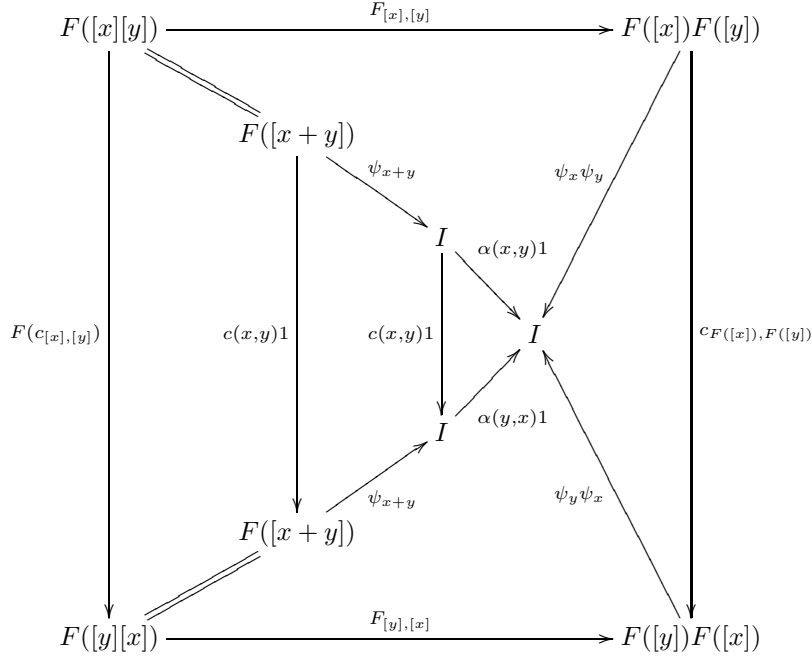
\xymatrix{
F([x][y]) \ar[rrrr]^{F_{[x],[y]}} \ar@{=}[dr] \ar[dddddd] _{F(c_{[x],[y]})} &&&& F([x])F([y]) \ar[dddddd]^{c_{F([x]),F([y])}} \ar[dddl]_{\psi_x\psi_y} \\
& F([x+y]) \ar[dr]^{\psi_{x+y}} \ar[dddd]_{c(x,y)1} \\
&& I \ar[rd]^(.3){\alpha(x,y)1}\ar[dd]_{c(x,y)1}\\
&&& I\\
&& I \ar[ru]_(.3){\alpha(y,x)1}\\
& F([x+y]) \ar[ru]_{\psi_{x+y}}\\
F([y][x]) \ar@{=}[ru] \ar[rrrr]^{F_{[y],[x]}} &&&& F([y])F([x]) \ar[uuul]^{\psi_y\psi_x}
}$$
\end{proof}

Now we construct a reduction $\overline F\se\cG(H/K,q)\to\cC$ of the functor $F$, i.e. a braided monoidal functor which fits in a commutative diagram of functors
\begin{equation}\label{fac}
\xymatrix{
\cG(H,c) \ar[rd]^F \ar[dd] \\
& \cC\\
\cG(H/K,q) \ar[ru]_{\overline F}
}
\end{equation}

\begin{prop}\label{fact}
Let $F\se\cG(H,c)\to\cC$ be a braided monoidal functor. Let $K\subset H$ be a subgroup and for each $x\in K$ let $\psi_x\se F([x])\to I$ be an isomorphism (with $\psi_0=1$). Let $s:H/K\to H$ be a section of the quotient map. Then the function $q:H/K\to k^\cro$, defined by $q(X) = c(s(X),s(X))$ is quadratic. Moreover, there is a braided monoidal functor $\overline F\se\cG(H/K,q)\to\cC$, making the diagram of functors (\ref{fac}) commutative. 
\end{prop}
\begin{proof}
Define the functor $\overline F$ by $\overline F([X]) = F([s(X)])$. Define the (pre-)monoidal structure $\overline F_{[X],[Y]}\se\overline F([X]\otimes[Y])\to\overline F([X])\otimes\overline F([Y])$ by 
$$
\xymatrix@!C=3.5pc{
& \overline F([X][Y]) \ar[rr]^{\overline F_{[X],[Y]}} && \overline F([X])\overline F([Y]) \ar@{=}[rd]\\
\overline F([X+Y]) \ar@{=}[ru] \ar@{=}[d] &&&& F([s(X)])F([s(Y)]) \\
\overline F([s(X+Y)]) \ar@{=}[rd] &&&& F([s(X)][s(Y)]) \ar[u]^{F_{[s(X)],[s(Y)]}}\\
& F([\theta(X,Y)+s(X)+s(Y)]) \ar@{=}[rd] & & F([\theta(X,Y)])F([s(X)][s(Y)]) \ar[ru]_{\psi_{\theta(X,Y)}1}\\
&& F([\theta(X,Y)][s(X)][s(Y)]) \ar[ru]_(.65){F_{[\theta(X,Y)],[s(X)][s(Y)]}}
}$$ 
Here $\theta\se H/K\otimes H/K\to K$ is the 2-cocycle associated with the section $s$: $\theta(X,Y) = s(X+Y)-s(X)-s(Y)$. 

Note that to construct a braided monoidal structure on $\cG(H/K)$, fitting in the diagram (\ref{fac}), it is enough to have a pair of functions: $a:H/K\times H/K\times H/K\to k^\cro$ and $\overline c:H/K\times H/K\to k^\cro$ such that the diagrams 
$$\xymatrix@!C=7.5pc{
\overline F([X][Y][Z]) \ar[r]^{\overline F_{[X],[Y][Z]}} \ar[dd]_{a(X,Y,Z)1} & \overline F([X])\overline F([Y][Z]) \ar[rd]^{1\overline F_{[Y],[Z]}} \\ 
&& \overline F([X])\overline F([Y])\overline F([Z]) \\
\overline F([X][Y][Z]) \ar[r]^{\overline F_{[X][Y],[Z]}} & \overline F([X][Y])\overline F([Z]) \ar[ru]_(.6){\overline F_{[X],[Y]}1}
}$$
$$\xymatrix{
\overline F([X][Y]) \ar[d]_{\overline c(X,Y)1} \ar[rr]^{\overline F_{[X],[Y]}} && \overline F([X])\overline F([Y]) \ar[d]^{c_{\overline F([X]),\overline F([Y])}}\\
\overline F([Y][X]) \ar[rr]^{\overline F_{[Y],[X]}} && \overline F([Y])\overline F([X])
}$$ commute. 
Indeed, pentagon and hexagon axioms for $a,\overline c$ will be fulfilled automatically. 

By substituting our choice for $\overline F_{[X],[Y]}$ in to the two diagrams above we get the following answers for $a,\overline c$:
$$a(X,Y,Z) = c(s(X),\theta(Y,Z))\alpha(\theta(X,Y+Z),\theta(Y,Z))\alpha(\theta(X+Y,Z),\theta(Y,Z))^{-1},$$
$$\overline c(X,Y) = c(s(X),s(Y)).$$
In particular the quadratic function associated to the pair $a,\overline c$ is $$q(X) = c(s(X),s(X))$$
\end{proof}

In the remaining part  of this section we will characterise those graded categories, which are tensor products of the trivial degree component and a category of graded vector spaces.

We call two objects $X$ and $Y$ in a braided monoidal category {\em mutually transparent} if and only if the double braiding is trivial on them:
$$c_{Y,X}c_{X,Y} = 1.$$
Two subcategories $\cA$ and $\cB$ of a braided monoidal category are {\em mutually transparent} if and only if for any $X\in\cA$ and $Y\in\cB$ $X$ and $Y$ are mutually transparent objects. 

\begin{theo}\label{grgr}
Let $\cC = \oplus_{X\in G/K}\cC_X$ be a $G/K$-graded braided monoidal category. Suppose that there exists a braided monoidal functor $F:\cG(G,c)\to\cC$ such that $F([x])\in\cC_x$ (here we identify $x$ with its coset modulo $K$) and such that $F([x])$ being transparent with $\cC_e$. Suppose also that there are isomorphisms $\psi_x:F([x])\to I$ with the identity object for every $x\in K$. Then as braided monoidal category $$\cC\simeq \cC_e\boxtimes\cG(G/K,q),$$ where $q$ is defined as in proposition \ref{fact}. 
\end{theo}
\begin{proof}
By proposition \ref{fact} the monoidal functor $F:\cG(G,c)\to\cC$ together with isomorphisms $\psi$ induces a braided monoidal functor $\overline F:\cG(G/K,c)\to\cC$, which together with the braided monoidal embedding $\cC_e\to\cC$ gives a grading preserving braided monoidal functor $\cC_e\boxtimes\cG(G/K,q)\to\cC.$ To see that it is an equivalence it is enough to note that for an arbitrary $x\in G/K$ any object $X$ of $\cC_x$ can be written as $X\otimes\overline F([x])^{-1}\otimes\overline F([x])$ and that $X\otimes\overline F([x])^{-1}$ belongs to $\cC_e$. 
\end{proof}

\end{document}